\documentclass[11pt]{amsart}
\usepackage{amsmath, amsthm, amscd, amsfonts}
\allowdisplaybreaks 
\usepackage{tikz}
\usepackage{pstricks,pst-node}
\usepackage[all]{xy}

\makeatletter \oddsidemargin.9375in \evensidemargin \oddsidemargin
\newtheorem{theorem}{Theorem}[section]
\newtheorem{lemma}[theorem]{Lemma}
\newtheorem{proposition}[theorem]{Proposition}

\theoremstyle{definition}
 
\newtheorem{example}[theorem]{Example}

\theoremstyle{remark}

\numberwithin{equation}{section}

\begin{document}

\title[Modulation preserving operators on LCA Groups ]{Modulation preserving operators  on
  Locally Compact Abelian Groups}

\author[M. Mortazavizadeh, R. Raisi Tousi ]{M. Mortazavizadeh, R. Raisi Tousi$^*$\\
January 13, 2019}

\subjclass[2010]{Primary  47A15 ; Secondary  42B99, 22B99.}

\keywords{locally compact abelian group, modulation invariant space,
range function, translation preserving operator, modulation preserving operator, range operator.}

\begin{abstract}
Let $G$ be a locally compact abelian group and $\Lambda$ be a closed subgroup of the dual group $\widehat{G}$. In this paper we investigate modulation preserving operators with respect to $\Lambda$, and give a characterization of them in terms of range operators.
\end{abstract} \maketitle

\section{Introduction and Preliminaries}
\noindent

 For a locally compact abelian (LCA) group  $G$, a translation invariant space is defined to be a closed subspace of $L^2(G)$  that is invariant under translations by elements of a closed subgroup  $\Gamma$ of $G$. Translation invariant spaces have been extensively grown in the theory and applications \cite{BHP, B, BR, Cab, KRr, KRs}.
% Translation invariant spaces in case of $\Gamma$ closed, discrete and cocompact, called shift invariant spaces, have been studied in \cite{Cab, KRr, KRs}, and extended to the case of $\Gamma$ closed and cocompact (but not discrete) in \cite{BR}, see also \cite{HelS}. Recently, translation invariant spaces have been generalized  in \cite{BHP} to the case when $\Gamma$  is closed (not necessarily discrete or cocompact). Translation preserving operators are bounded linear operators that commute with the translation operator.
Bownik in \cite{B} gave a characterization of shift preserving operators on $L^2(\Bbb{R}^d)$ in terms of range operators. In the setting of LCA groups, a bounded linear operator on $L^2(G)$ is said to be shift preserving if it commutes with translations by elements of a closed subgroup $\Gamma$ of $G$ which is discrete and cocompact. These operators have been studied in \cite{KRsh}. In \cite{MRK}, we investigated translation preserving operators, that is operators commuting with translations by elements of a closed subgroup of $G$ which is not necessarily discrete or cocompact. Another spaces, which are effective tools in Gabor theory, are spaces invariant under modulations. In \cite{MR}, using a range function approach, we studied modulation invariant spaces. We define a closed subspace $W$ of  $L^2(G)$  to be modulation invariant, if it is invariant under modulations by elements of a closed subgroup $\Lambda$ of $\widehat{G}$ which is not necessarily discrete or cocompact. In \cite{MR},  we proved that there is a one to one correspondence between modulation invariant spaces and range functions. Our goal in this paper is to investigate modulation preserving operators. We define a modulation preserving operator as a bounded linear operator on $L^2(G)$ which commutes with the modulation operator. We give a characterization of modulation preserving operators in terms of range operators. We first show that there is a one to one correspondence between modulation preserving operators on $L^2(G)$ and multiplication preserving operators on a vector valued space. We  then use this correspondence to get the characterization of modulation preserving operators in terms of range operators. Furthermore, we show that a modulation preserving operator has several properties in common with the associated range operator, especially compactness of one implies compactness of the other. We obtain a necessary condition for a modulation preserving operator to be  Hilbert Schmidt or of finite  trace. We have organized the article as follows. The rest of this section is devoted to stating some required preliminaries on translation invariant spaces and translation preserving operators which were studied in \cite{BR} and \cite{MRK}. Section 2 contains the main results of the paper.  Using a transformation of $L^2(G)$ into a vector valued space, we find a correspondence between modulation preserving operators on $L^2(G)$ and multiplication preserving operators on the vector valued space, which yields the desired characterization. Finally, we find relations between some properties of modulation preserving operators and the corresponding range operators. For a modulation preserving operator $U$, we show that if $U$ is Hilbert Schmidt (of finite trace), then so is the range operator associated to $U$.
 
Let $(\Omega, m)$ be a $\sigma$- finite measure space and $\mathcal{H}$ be a separable Hilbert space. A range funtion is a mapping $J: \Omega \longrightarrow \lbrace \textrm{  closed subspaces  of   $\mathcal{H}$ } \rbrace$. We write $P_{J}(\omega)$ for the orthogonal projections of $\mathcal{H}$ onto $J(\omega)$. A range function $J$ is measurable if the mapping $\omega \mapsto \langle P_{J}(\omega)(a) , b \rangle$ is measurable for all $a ,b \in \mathcal{H} $. Consider the space $L^{2}(\Omega, \mathcal{H})$ of all measurable functions $\phi$ from $\Omega$ to $\mathcal{H}$ such that $\Vert \phi \Vert _{2}^{2} = \int_{\Omega} \Vert \phi(\omega)\Vert^{2}_{\mathcal {H}} dm(\omega) <\infty$ with the inner product $\langle \phi , \psi \rangle = \int_{\Omega} \langle \phi(\omega) , \psi(\omega) \rangle_{\mathcal{H}} d m(\omega)$. A subset $\mathcal{D}$ of $L^{\infty}(\Omega) $ is said to be a determinig set for $L^{1}(\Omega) $, if for all $ f\in L^{1}(\Omega)$, $\int_{\Omega} fg dm =0 $ for all $ g \in \mathcal{D}$ implies that $f=0$. A closed subspace $\mathcal{W}$ of $L^{2}(\Omega, \mathcal{H})$ is called multiplicatively invariant with respect to a determining set $\mathcal{D}$, if for each $\phi \in \mathcal{W}$ and $g \in \mathcal{D}$ one has $g\phi \in \mathcal{W}$.
 Bownik and Ross in \cite[Theorem 2.4]{BR} showed that there is a correspondence between  multiplicatively invariant spaces and measurable range functions as follows.
\begin{proposition} \label{p1.1}
Suppose that $L^{2}(\Omega)$ is separable, so that $L^{2}(\Omega, \mathcal{H})$ is also separable. Then for a closed subapace $\mathcal{W}$ of $L^{2}(\Omega, \mathcal{H})$ and a determining set $\mathcal{D}$  for $L^{1}(\Omega) $ the following are equivalent.\\
(1)  $\mathcal{W}$ is multiplicatively invariant with respect to $\mathcal{D}$.\\
(2)  $\mathcal{W}$ is multiplicatively invariant with respect to $L^{\infty}(\Omega)$.\\
(3) There exists a measureble range function $J$ such that 
\begin{equation*}
\mathcal{W}=\lbrace \phi \in L^{2}(\Omega , \mathcal{H}) : \phi(\omega) \in J(\omega) \ \text{,} \ \  \text{ a.e. }  \omega \in \Omega \rbrace .
\end{equation*}
Identifying range functions which are equivalent a.e., the correspondence between $\mathcal{D}$- multiplicatively invariant spaces and measurable range functions is one to one and onto. Moreover, there is a countable subset $\mathcal{A}$ of  $L^{2}(\Omega, \mathcal{H})$ such that $\mathcal{W}$ is the smallest closed $\mathcal{D}$- multiplicatively invariant space containing $\mathcal{A}$. For any such $\mathcal{A}$ the measurable range function associated to $\mathcal{W}$ satisfies
\begin{equation*}
J(\omega) = \overline{span} \lbrace \phi (\omega) : \phi \in \mathcal{A}\rbrace \ \  a.e. \  \omega \in \Omega .  \label{J} 
\end{equation*}
\end{proposition}
Let $\mathcal{W}$ be a multiplicatively invariant space with respect to a determining set $\mathcal{D}$ and $J$ be the corresponding  range function. In \cite{MRK} we defined a range operator on $J$ as a mapping $R$ from $\Omega$  to the set af all bounded linear operators on closed subspaces of $\mathcal{H}$. We also defined a $\mathcal{D}$- multiplication preserving operator on $\mathcal{W}$ as an operator $U: \mathcal{W} \longrightarrow  L^2(\Omega ,  \mathcal{H})$ such that for all $g \in \mathcal{D}$ and $\phi \in \mathcal{W}$ 
\begin{equation} \label{r.op}
U(g \phi)(\omega) = g(\omega) U(\phi)(\omega), \ \ \omega \in \Omega.
\end{equation}
The following proposition presents a characterization of multiplication preserving operators in terms of range operators for which we sketch a proof. For more details see \cite[Theorem 2.2]{MRK}.
\begin{proposition} \label{t3.3}
Suppose that $\mathcal{W} \subseteq L^2(\Omega ,  \mathcal{H})$ is a multiplicatively invariant space with respect to a determining set $\mathcal{D}$. Assume that $U: \mathcal{W} \longrightarrow  L^2(\Omega ,  \mathcal{H})$ is a bounded linear operator. Then the following are equivalent.\\
(1)  $U$ is  multiplication preserving  with respect to $\mathcal{D}$.\\
(2)  $U$ is  multiplication preserving  with respect to $L^{\infty}(\Omega) $.\\
(3) There exists a measurable range operator $R$ on $J$ such that for all $\phi \in \mathcal{W}$,
\begin{equation*}
U\phi(\omega) = R(\omega)(\phi(\omega)) \ \ a.e. \ \omega \in \Omega .
\end{equation*} 
Conversely, given a measurable range operator $R$ on $J$ with $ess \sup \Vert R(\omega) \Vert < \infty$,  there is a bounded multiplication preserving operator $ U : \mathcal{W} \longrightarrow L^2(\Omega ,  \mathcal{H})$, such that \eqref{r.op} holds. The correspondence between $U$ and $R$ is one to one under the usual convention that the range operators are identified if they are equal a.e. Moreover $\Vert U \Vert = ess  \sup \Vert R(\omega) \Vert $.
\end{proposition}
\begin{proof}
The implications $(3)\Rightarrow (2)$ and $(2)\Rightarrow(1)$ are obvious. Suppose that (1) holds. For $g \in L^{\infty}(\Omega)$ let $M_g$ be the multiplication operator $ M_g : L^2(\Omega, \mathcal{H}) \longrightarrow L^2(\Omega, \mathcal{H})$, $ M_{g} \phi (\omega) = g(\omega) \phi(\omega)$. Using the embedding $g\mapsto M_g$, we can consider $\mathcal{D}$ and $ L^{\infty}(\Omega)$ as subsets of $\mathcal{B}(L^2(\Omega, \mathcal{H}))$, the set of all bounded linear operators on $L^2(\Omega, \mathcal{H})$. In the context of von Neumann algebras since (1) holds, we have $\mathcal{D} \subseteq \mathcal{R}(U)^{'}$, where $\mathcal{R}(U)$ is the von Neumann algebra generated by $U$. We show that  $L^{\infty}(\Omega) \subseteq \mathcal{R}(U)^{'}$. It is enough to show that $\mathcal{D}^{'} \subseteq L^{\infty}(\Omega) ^{'}$. Suppose by contradiction that $ x \in \mathcal{D}^{'} \setminus L^{\infty}(\Omega) ^{'} $. Let $y_0 \in L^{\infty}(\Omega) $ be such that $xy_0 - y_{0}x \neq 0$. Now by \cite[Theorem 4.2.3]{M}, there exists a weak operator continuous operator $\omega_{\xi , \eta} \in L^1 (L^2(\Omega, \mathcal{H}))$  such that 
 \begin{equation} \label{T}
\omega_{\xi , \eta}(xy_0 - y_{0}x) \neq 0,
\end{equation}
where $\xi , \eta \in \mathcal{H}$ and  for $T \in \mathcal{B}(L^2(\Omega, \mathcal{H}))$, 
 \begin{equation*} 
\omega_{\xi , \eta}(T) = \langle T \xi , \eta \rangle.
\end{equation*}

 We show that if $\lbrace f_{\alpha} \rbrace$ is a net in $L^{\infty}(\Omega)$ converges to $f$ in the weak-* topology, then $M_{f_{\alpha}} \rightarrow M_{f}$ in the weak operator topology. Using the fact that $L^2(\Omega, \mathcal{H}) \cong L^2 (\Omega) \otimes \mathcal{H}$,  we have for $\varphi_0 , \varphi _1 \in L^2 (\Omega) $ and $\xi _ 0 , \xi _1 \in \mathcal{H}$ 
\begin{align*}
\langle M_{f_{\alpha}} (\varphi_0 \otimes \xi_0 ) , \varphi_1 \otimes \xi_1 \rangle &= \langle f_{\alpha}(\varphi_0 \otimes \xi_0 ), \varphi_1 \otimes \xi_1 \rangle \cr
&= \langle f_{\alpha}\varphi_0 , \varphi_1 \rangle \langle \xi_0 , \xi_1  \rangle \cr
&= \langle \xi_0 , \xi_1  \rangle \int_{\Omega} f_{\alpha}(\omega) \varphi_0 (\omega) \varphi_1 (\omega)d\omega 
\end{align*}
Now since $\varphi_0 \varphi_1 \in L^1 (\Omega)$, so 
\begin{equation*}
\int_{\Omega} f_{\alpha}(\omega) \varphi_0 (\omega) \varphi_1 (\omega)d\omega  \rightarrow \int_{\Omega} f(\omega) \varphi_0 (\omega) \varphi_1 (\omega)d\omega ,
\end{equation*}
thus
\begin{equation*}
\langle M_{f_{\alpha}} (\varphi_0 \otimes \xi_0 ) , \varphi_1 \otimes \xi_1 \rangle \rightarrow \langle M_{f} (\varphi_0 \otimes \xi_0 ) , \varphi_1 \otimes \xi_1 \rangle .
\end{equation*}
Now define $ F : L^{\infty}(\Omega) \longrightarrow \mathbb{C}$, given by
\begin{equation*}
F(f) = \omega_{\xi , \eta}( x M_f - M_f x).
\end{equation*}
By weak operator continuity of $\omega_{\xi , \eta}$ it follows that $F$ is weak-* continuous and thus $F \in L^1 (\Omega)$. Since $x \in \mathcal{D}^{'}$, we have $F\mid_{\mathcal{D}} =0 $ and hence $F=0$ which is a contradiction to \eqref{T} and proves (2). 

Now let (2) hold. Assume that $\mathcal{A}$ is a countable subset of $L^2(\Omega ,  \mathcal{H})$ which generates $W$. By Proposition \ref{p1.1} 
\begin{equation*}
 J(\omega) = \overline{span} \lbrace \phi (\omega) : \phi \in \mathcal{A}\rbrace \ \  a.e. \  \omega \in \Omega .
\end{equation*}
We define the operator $S(\omega) $ on the set $\lbrace \phi (\omega) :  \phi \in \mathcal{A}\rbrace$  by
\begin{equation} \label{r.o}
S(\omega) (\phi (\omega)) = U \phi (\omega).
\end{equation}
Then $S(\omega)$ is clearly linear. Also  for any $g \in L^{\infty}(\Omega)$ we have
\begin{align*}
\int_{\Omega} \vert g(\omega) \vert ^2 \Vert U \phi(\omega) \Vert ^2 d\omega &= \Vert g U \phi  \Vert_2 ^ 2 \\
&= \Vert  U g\phi  \Vert_2 ^ 2 \\
& \leq \Vert U \Vert ^2 \Vert  g\phi  \Vert _2 ^ 2 \\
&=  \Vert U \Vert ^2 \int_{\Omega} \vert g(\omega) \vert ^2 \Vert  \phi(\omega) \Vert ^2 d\omega. 
\end{align*}
Since $g \in L^{\infty}(\Omega)$ is arbitrary, this implies that 
\begin{equation} \label{ufi}
\Vert U \phi(\omega) \Vert \leq \Vert U \Vert \Vert \phi(\omega) \Vert, a.e.\  \omega \in \Omega,
\end{equation}
which shows that $S(\omega)$ is bounded and hence it is extended to a bounded linear operator $R(\omega)$ on $  \overline{span} \lbrace \phi (\omega) : \phi \in \mathcal{A} \rbrace = J(\omega) $, as is desired in (3). For the moreover part, \eqref{ufi} clearly implies that $ess\sup_{\omega \in \Omega} \Vert R(\omega) \Vert \leq \Vert U \Vert$. Also we have
\begin{align*}
\Vert U \phi \Vert _2 ^2  &= \int _{\Omega} \Vert U \phi (\omega) \Vert  ^2 d\omega \\
&=  \int _{\Omega} \Vert R(\omega) \phi (\omega) \Vert  ^2 d\omega \\
& \leq \int _{\Omega} \Vert R(\omega) \Vert ^2 \Vert \phi (\omega) \Vert  ^2 d\omega \\
& \leq ess\sup_{\omega \in \Omega} \Vert R(\omega) \Vert ^2   \int _{\Omega}  \Vert \phi (\omega) \Vert  ^2 d\omega \\
&= ess\sup_{\omega \in \Omega} \Vert R(\omega) \Vert ^2 \Vert \phi \Vert _2^2 .
\end{align*}
Thus  $\Vert U \Vert = ess\sup_{\omega \in \Omega} \Vert R(\omega) \Vert $. Finally, by \eqref{r.op} the correspondence between $U$ and $R$ is one to one and onto.
\end{proof}
Assume that $G$ is a second countable LCA group, $\Gamma$ is a closed subgroup of $G$, and the annihilator of $\Gamma$ in $\widehat{G}$ is defined as $\Gamma ^* = \lbrace \chi \in \widehat{G} : \chi(\gamma)=1, \ \gamma \in \Gamma \rbrace$. Suppose also that $\Omega$ is a measurable section for the quotient $\widehat{G} / \Gamma ^{*}$ and $C$ is a measurable section for the quotient $G / \Gamma$. For $\gamma \in \Gamma$ we denote by $X_{\gamma}$ the corresponding character on $\widehat{G}$, i.e. $X_{\gamma}(\chi )= \chi(\gamma)$ for  $\chi \in \widehat{G}$. One can see that the set $ \mathcal{D}= \lbrace X_{\gamma} \vert_{\Omega} : \gamma \in \Gamma \rbrace $ is a determining set for $L^{1}(\Omega) $. A closed subspace $ V \subseteq L^{2}(G)$ is called $\Gamma$- translation invariant space, if $T_{\gamma} V \subseteq V $ for all $\gamma \in \Gamma$. We say that $V$ is generated by a countable subset $\mathcal{A}$ of  $L^{2}(G)$, when $V=S^{\Gamma}(\mathcal{A})=\overline{span}\lbrace T_{\gamma}f : f \in \mathcal{A} , \gamma \in \Gamma \rbrace$.
In \cite[Proposition 6.4]{BHP} it is shown  that there exists an isometric isomorphism, namely Zak tranasform $Z: L^{2} (G) \longrightarrow L^{2}(\Omega , L^{2}(C))$  satisfying 
 \begin{equation} \label{Zak}
 Z(T_{\gamma}\phi )= X_{\gamma} \vert_{\Omega} Z (\phi).
 \end{equation}
  A bounded linear operator $U$ on $L^2(G)$ is said to be translation preserving with respect to a closed subgroup $\Gamma$ of $G$, if $UT_{\gamma} = T_{\gamma}U$, where $T_{\gamma}$ is the translation oparator. Let $U$ be a translation preserving operator on a translation invariant space $V$. We define  an induced operator $U'$ on the multiplicatively invariant space $Z(V)$ as
\begin{equation} \label{io}
U': Z(V) \longrightarrow  L^2(\Omega ,  L^2(C))  , \ U'(Zf)= Z(Uf),
\end{equation}
where $Z$ be as in \eqref{Zak}.
Note that it is easy to see that 
\begin{equation} \label{zmp}
  U'(X_{\gamma}  Zf)(\omega) = X_{\gamma}(\omega)U'(Zf)(\omega), 
 \end {equation}
 which implies that $U'$ is a multiplication preserving operator(see also \cite{MRK}).
Let $V$ be a  translation invariant space with the range function $J$. A range operator on $J$ is a mapping $R$ from the Borel section $\Omega$ of $\widehat{G}/ \Gamma ^* $ to the set of all bounded linear operators on closed subspaces of $L^{2}(C)$, where $C$ is a Borel section of $G/ \Gamma$, so that the domain of $R(\omega)$ equals $J(\omega)$ for a.e. $\omega \in \Omega$. A range operator $R$ is called measurable, if the mapping $\omega \mapsto \langle R(\omega)P_{J}(\omega)(a) , b \rangle$ is measurable for all $a,b \in L^2(C)$, where $P_{J}(\omega)$ is the orthogonal projection of $L^{2}(C)$ onto $J(\omega)$.  Now we obtain the following proposition which characterizes  translation preserving operators in terms of range operators, see also \cite[Theorem 2.5]{MRK}.
\begin{proposition} \label{t3.6}
Let $ V \subseteq L^{2}(G)$ be a $\Gamma$- translation invatiant subspace with the range function $J$ and $U: V \longrightarrow L^2(G)$ be a bounded linear operator. Then the following are equivalent. \\
(1) The operator $U$ is translation preserving with respect to $\Gamma$. \\
(2) The induced operator $U'$ is multiplication preserving operator with respect to $L^{\infty}(\Omega) $.\\
(3) There exists a measurable range operator $R$ on $J$ such that for all $\phi \in V$ 
\begin{equation*}
Z U\phi(\omega) = R(\omega)(Z\phi(\omega)) \ \ a.e. \ \omega \in \Omega .
\end{equation*} 
The correspondence between $U$ and $R$ is one to one under the usual convention that the range operators are identified if they are equal a.e.. Moreover $\Vert U \Vert = ess  \sup \Vert R(\omega) \Vert $.
\end{proposition}
\begin{proof}
By the fact that the induced operator $U'$ is multiplication preserving, the implication $(1) \Rightarrow (2)$ is obvious. Using Proposition \ref{t3.3}, one can easily show that (2) implies (3). Assume that (3) holds. For $\gamma \in \Gamma$ and $\phi \in V$, we have
\begin{align*}
Z(U T_{\gamma} \phi)(\omega) &= R(\omega)(Z(T_{\gamma} \phi)(\omega)) \cr
&= R(\omega)(X_{\gamma}(\omega) Z\phi(\omega)) \cr
&= X_{\gamma}(\omega) R(\omega)(Z\phi(\omega)) \cr
&= X_{\gamma}(\omega) Z U(\phi)(\omega) \cr
&= Z( T_{\gamma} U \phi)(\omega).
\end {align*}
Now (1) follows from the fact that $Z$ is one to one. Proposition \ref{t3.3} implies also that the correspondence between $R$ and $U$ is unique. The moreover part follows from Proposition \ref{t3.3} and the fact that $\Vert U \Vert = \Vert U' \Vert$.
\end{proof}

%\begin{proof}
%Using the fact that unitary operators preserve frames and Riesz bases \cite[Section 5.3]{Ole}, we know that $E^{\Lambda}(\mathcal{A})$ is a
%continuous
% frame for $ M^{\Lambda}(\mathcal{A})$, if and only if
%\begin{equation*}
%\mathcal{F}(E^{\Lambda}(\mathcal{A})) = \lbrace T_{\lambda} \hat{\phi} \ : \ \phi \in \mathcal{A} \rbrace
%\end{equation*}
% is a   continuous  frame (Riesz basis) for
%\begin{equation*}
% \mathcal{F}(M^{\Lambda}(\mathcal{A})) = \overline{span} \lbrace T_{\lambda} \hat{\phi} \ : \ \phi \in \mathcal{A} \rbrace .
%\end{equation*}
% Equivalently, for almost every $x \in \Pi$, the set  $\lbrace Z \hat{\phi} (x) \ : \ \phi \in \mathcal{A} \rbrace$ is a frame (Riesz basis) for $J(x)$, where $Z$ is as in \eqref{Zak} (see  \cite[Theorem 6.10]{BHP}). Now by \eqref{ztild},
%\begin{equation*}
%\lbrace Z \hat{\phi} (x) \ : \ \phi \in \mathcal{A} \rbrace = \lbrace \tilde{Z}{\phi} (x) \ : \ \phi \in \mathcal{A} \rbrace .
%\end{equation*}
%This completes the proof.
%\end{proof}

\section{Modulation preserving operators}

Let $\Lambda$ be a closed subgroup of $\widehat{G}$. Assume that  $\Lambda^{*}$ is the annihilator of $\Lambda$ in $G$, i.e. $\Lambda ^* = \lbrace x \in G : \lambda(x)=1, \ \lambda \in \Lambda \rbrace$. In addition, suppose that $\Pi$ is a measurable section for the quotient $G / \Lambda ^{*}$ and $D$ is a measurable section for the quotient $\widehat{G} / \Lambda$. For $\lambda \in \Lambda$ we denote by $X_{\lambda}$ the corresponding character on $G$. One can see that the set $ \mathcal{D}= \lbrace X_{\lambda} \vert_{\Pi} : \lambda \in \Lambda \rbrace $ is a determining set for $L^{1}(\Pi) $. A closed subspace $ W \subseteq L^{2}(G)$ is called $\Lambda$- modulation invariant space, if $M_{\lambda} W \subseteq W $ for all $\lambda \in \Lambda$, where $M_{\lambda}$ is the modulation operator defined as $M_{\lambda} : L^{2} (G) \longrightarrow L^{2} (G)$, $M_{\lambda}f(x)=\lambda(x)f(x)$.
 Let $\mathcal{F}$ denote the Fourier transform and $Z$ be the Zak transform. We define an isometric isomorphism as
 \begin{equation} \label{ztild}
  \tilde{Z} : L^2(G) \longrightarrow L^{2}(\Pi , L^{2}(D))  , \  \tilde{Z}:= Z \ o\ \mathcal{F} .
 \end{equation}
 In the following proposition, we obtain a characterization of frame and Riesz sequence property for modulation invariant systems in terms of the Zak transform. In the case of translation invariant spaces, this formolation was done in \cite[Theorem 6.6]{BHP}. Using the fact that unitary operators preserve frames and Riesz bases (see \cite[Section 5.3]{Ole}), the proof of the following proposition is a direct consequence of the case of translation, so we omit the  proof. Also, we gave a complete proof in \cite[Theorem 2.3]{MR}.
\begin{proposition} \label{fr}
Let $\mathcal{A} \subseteq L^2(G)$ be a countable subset and $J$ be the measurable range function associated to $W= M^{\Lambda}(\mathcal{A})$. Assume that $E^{\Lambda}(\mathcal{A}) := \lbrace M_{\lambda} \phi \ : \ \phi \in \mathcal{A} \rbrace$. The following conditions are equivalent.\\
(1) $E^{\Lambda}(\mathcal{A})$ is a  continuous  frame (continuous Riesz basis) for $W$ with bounds $ 0 < A \leq B < \infty $. \\
(2) The set $\lbrace \tilde{Z} \phi (x) \ : \ \phi \in \mathcal{A} \rbrace$  is a frame (Riesz basis) with bounds $A$ and $B$, for almost every $x \in \Pi$.
\end{proposition}
The next lemma states that every $\Lambda$- modulation invariant space can be decomposed to mutually orthogonal $\Lambda$- modulation invariant spaces each of which is generated by a single function in $L^2(G)$. The proof is similar to \cite[Theorem 5.3]{BR} and so  is omitted.
\begin{lemma} \label{t2.3}
Let $W$ be a $\Lambda$- modulation invariant subspace of $L^{2}(G)$. Then there exist functions $\phi_{n} \in W$, $n \in \mathbb{N}$ such that\\
(1) The set $\{ M_{\lambda} \phi _{n} \ : \ \lambda \in \Lambda  \}$ is a Paseval frame for $M^{\Lambda}(\phi _{n})$. \\
(2) The space $W$ can be decomposed as an orthogonal sum 
\begin{equation*}
W= \bigoplus _{n \in \mathbb{N}} M^{\Lambda}(\phi _{n}).
\end{equation*}
\end{lemma}

Note that $ \tilde{Z}$  turns $\Lambda$- modulation invariant  spaces in $L^{2} (G)$ into multiplicatively invariant spaces in $L^{2}(\Pi , L^{2}(D))$  and vice versa. Here we establish a characterization of $\Lambda$- modulation invariant  spaces in terms of range functions as follows. The proof is similar to the translation case (\cite[Theorem 6.5]{BHP}) and so is omitted (see also \cite{MR}).
\begin{proposition} \label{p1.2}
Let $W \subseteq L^{2}(G)$ be a closed subapace and $\tilde{Z}$ be as in \eqref{ztild}. Then the following are equivalent.\\
(1)  $W$ is  $\Lambda$- modulation invariant. \\
(2)  $\tilde{Z}(W)$ is a multiplicavely invariant subspace of $ L^{2}(\Pi , L^{2}(D))$ with respect to the determining set $ \mathcal{D}= \lbrace X_{\lambda} \vert_{\Pi} : \lambda \in \Lambda \rbrace $. \\
(3) There exists a measurable range function $J: \Pi \longrightarrow \lbrace  closed \ subspaces \  of \   L^{2}(D) \rbrace $ such that 
\begin{equation} \label{mi}
W=\lbrace f \in L^{2}(G) : \tilde{Z}(f)(x)\in J(x)  \text{,} \ \  \text{for a.e. }  x \in \Pi \rbrace .
\end{equation}
Identifying range functions which are equivalent a.e., the correspondence between $\Lambda$- modulation invariant spaces and measurable range funtions is one to one and onto. Moreover if $W=M^{\Lambda}(\mathcal{A})$ for some countable subset $\mathcal{A}$ of $L^{2}(G)$, the measurable range function $J$  associated to $W$ is given by
\begin{equation*}
J(x) = \overline{span} \lbrace \tilde{Z}(\phi) (x) : \phi \in \mathcal{A}\rbrace \ \  a.e.  \  x \in \Pi .  
\end{equation*}
\end{proposition}

A bounded linear operator $U$ on $L^2(G)$ is said to be modulation preserving with respect to $\Lambda$, if for every $\lambda \in \Lambda$, $UM_{\lambda} = M_{\lambda}U$, where $M_{\lambda}$ is the modulation oparator. Our goal in this section is a charactrization of $\Lambda$- modulation preserving operators in terms of range operators. Let $W$ be a $\Lambda$- modulation invariant space with the range function $J$. A range operator on $J$ is a mapping $R$ from the Borel section $\Pi$ of $G/ \Lambda ^* $ to the set of all bounded linear operators on closed subspaces of $L^{2}(D)$, where $D$ is a Borel section of $\widehat{G}/ \Lambda$, so that the domain of $R(x)$ equals $J(x)$ for a.e. $x \in \Pi$. A range operator $R$ is called measurable, if the mapping $x \mapsto \langle R(x)P_{J}(x)(a) , b \rangle$ is measurable for all $a,b \in L^2(D)$, where $P_{J}(x)$ is the orthogonal projection of $L^{2}(D)$ onto $J(x)$.
Suppose that $U$ is a $\Lambda$- modulation preserving operator on a $\Lambda$- modulation invariant space $W$ and $\tilde{Z}$ is as in \eqref{ztild}. We can define an induced functorial operator on the multiplicatively invariant space $\tilde{Z}(W)$ as
\begin{equation} \label{ioz}
U'' : \tilde{Z}(W) \longrightarrow  L^2(\Pi ,  L^2(D))  , \ U''(\tilde{Z} f)= \tilde{Z}(Uf).
\end{equation}
For $f \in W$,
\begin{eqnarray*}
U''(X_{\lambda}  \tilde{Z} f)(x)&=& U''(X_{\lambda} Z \hat{f})(x) \\
&=& U'' (Z(T_{\lambda} \hat{f}))(x) \\
&=& U'' (Z(\widehat{M_{\lambda}f})(x) \\
&=& U'' (\tilde{Z}(M_{\lambda}f))(x) \\
&=& \tilde{Z}(UM_{\lambda}f)(x) \\
&=& \tilde{Z}(M_{\lambda}Uf)(x) \\ 
&=& X_{\lambda}(x) \tilde{Z}(Uf)(x) \\
&=&X_{\lambda}(x)U''(\tilde{Z} f)(x),
\end{eqnarray*} 
where $X_{\lambda}$ is the corresponding character on $G$, for $\lambda \in \Lambda$. Consequently, for a $\Lambda$- modulation preserving operator $U$, the operator $U''$ defined in \eqref{ioz} is a multiplication preserving operator on $\tilde{Z}(W)$. 

In the sequel (Theorem \ref{t3.6}), we apply Proposition \ref{t3.3} to the operator $U''$ so that we can characterize the modulation preserving operator $U$. We need the following lemma which shows that using the Fourier transform, we can transform modulation preserving operators on $L^2(G)$ into translation preserving operators on $L^2(\widehat{G})$.
\begin{lemma} \label{remark}
For a $\Lambda$- modulation preserving operator $U: W \longrightarrow L^2(G)$ on a $\Lambda$- modulation invariant space $W$, the operator \begin{equation*} 
\mathcal{F} \ o \ U \ o \ \mathcal{F}^{-1}  : \mathcal{F}(W) \longrightarrow  L^2(\widehat{G})
\end{equation*}
is  $\Lambda$- translation preserving. Moreover for $ f \in W$, 
 \begin{equation} \label{tarkib}
 U'' (\tilde{Z} f)(x) = ( \mathcal{F} \ o \ U \ o \ \mathcal{F}^{-1})'(Z \hat{f}(x)), 
 \end{equation}
 in which $U''$ is as in \eqref{ioz}.
\end{lemma}
\begin{proof}
For $\lambda \in \Lambda$ we have the following calculations
\begin{eqnarray*}
(\mathcal{F} \ o \ U \ o \ \mathcal{F}^{-1} ) T _{\lambda} &=& \mathcal{F} \ o \ U \ o \ M_{\lambda} \ o \ \mathcal{F}^{-1} \\
&=& \mathcal{F} \ o \  M_{\lambda} \ o \ U \ o \  \mathcal{F}^{-1} \\
&=& T _{\lambda} ( \mathcal{F} \ o \ U \ o \ \mathcal{F}^{-1}) .
\end{eqnarray*}
Let $Z$ be the Zak transform. By \eqref{zmp}, there exists an induced operator 
\begin{equation*} 
( \mathcal{F} \ o \ U \ o \ \mathcal{F}^{-1})' : Z(\mathcal{F}(W)) \longrightarrow  L^2(\Pi ,  L^2(D))  
\end{equation*}
given by
 \begin{equation} \label{2.1.}
 ( \mathcal{F} \ o \ U \ o \ \mathcal{F}^{-1})'(Z \hat{f})= Z( \mathcal{F} \ o \ U \ o \ \mathcal{F}^{-1}\hat{f}) \ , \ f \in W
\end{equation}
 which is multiplication preserving. For the moreover part, we have for $ f \in W$, 
\begin{eqnarray*}
U'' (\tilde{Z} f)(x) &=& \tilde{Z}(U f)(x) \\
&=& (Z \ o \ \mathcal{F} \ o \ U) f (x ) \\
 & =& Z (\mathcal{F} \ o \ U \ o \ \mathcal{F}^{-1}) \hat{f}(x) \\
 & =& ( \mathcal{F} \ o \ U \ o \ \mathcal{F}^{-1})'(Z \hat{f}(x)).
\end{eqnarray*}
\end{proof}
 Now we can characterize modulation preserving operators on $L^2(G)$ in terms of range operators.
\begin{theorem} \label{t3.6}
Let $ W \subseteq L^{2}(G)$ be a $\Lambda$- modulation invatiant subspace with the range function $J$ and $U: W \longrightarrow L^2(G)$ be a bounded linear operator. Then the following are equivalent. \\
(1)  $U$ is modulation preserving with respect to $\Lambda$. \\
(2) The induced operator $U''$ is multiplication preserving.\\
(3) There exists a measurable range operator $R$ on $J$ such that for all $\phi \in W$,
\begin{equation*}
\tilde{Z} (U\phi)(x) = R(x)(\tilde{Z} \phi(x)) \ \ a.e. \ x \in \Pi .
\end{equation*} 
The correspondence between $U$ and $R$ is one to one under the usual convention that the range operators are identified if they are equal a.e.. Moreover $\Vert U \Vert = ess  \sup \Vert R(\omega) \Vert $.
\end{theorem}
\begin{proof}
For  $(1) \Rightarrow (2)$ assume that $U$ is a $\Lambda$- modulation preserving operator. By Lemma \ref{remark}, the operator $\mathcal{F} \ o \ U \ o \ \mathcal{F}^{-1}$ is a  $\Lambda$- translation preserving operator on $\mathcal{F}(W)$. So the induced operator $ (\mathcal{F} \ o \ U \ o \ \mathcal{F}^{-1})'$ defined as \eqref{io} is a multiplication preserving operator on the muliplicatively invariant space $\tilde{Z}(W)$. By \eqref{tarkib}, 
 \begin{eqnarray*}
U'' \tilde{Z} f (x) &=& ( \mathcal{F} \ o \ U \ o \ \mathcal{F}^{-1})'Z \hat{f}(x)\\
&=& ( \mathcal{F} \ o \ U \ o \ \mathcal{F}^{-1})' \tilde{Z}f(x), 
\end{eqnarray*}
which proves $(2)$. For $(2) \Rightarrow (3)$, suppose that $U''$ is a multiplication preserving operator. By Proposition \ref{t3.3}, there exists a measurable range operator $R$ on $J$ such that for all $\phi \in W$,
\begin{equation*}
U''(\tilde{Z} \phi)(x) = R(x) (\tilde{Z} \phi (x)).
\end{equation*}
Using \eqref{tarkib}, we obtain 
\begin{eqnarray*}
\tilde{Z}(U \phi)(x)&=& Z \ o \  \mathcal{F} (U \phi ) (x) \\
&=& Z(\mathcal{F} \ o \ U \ \mathcal{F}^{-1} \hat{\phi} ) (x) \\
&=& (\mathcal{F} \ o \ U \ o \ \mathcal{F})^{'} (Z\hat{\phi}(x)) \\
&=& U^{''} (\tilde{Z} \phi )(x) \\
&=&  R(x) (\tilde{Z} \phi (x)),
\end{eqnarray*}
which proves $(3)$.
Now assume that (3) holds. For $\lambda \in \Lambda$ and $\phi \in W$,
\begin{eqnarray*}
\tilde{Z }(U M_{\lambda} \phi)(x) &=& R(x)(\tilde{Z}(M _{\lambda} \phi)(x)) \\
&=& R(x)(X_{\lambda}(x) \tilde{Z}\phi(x)) \\
&=& X_{\lambda}(x) R(x)(\tilde{Z}\phi(x)) \\
&=& X_{\lambda}(x) \tilde{Z} U(\phi)(x) \\
&=& \tilde{Z}( M_{\lambda} U \phi)(x).
\end {eqnarray*}
Then (1) follows from the fact that $\tilde{Z}$ is one to one. Proposition \ref{t3.3} implies that the correspondence between $R$ and $U$ is unique. The moreover part follows from Proposition \ref{t3.3} and the fact that $\Vert U \Vert = \Vert U'' \Vert$.
\end{proof}

We now establish some relations between modulation preserving operators and their range operators. We show that compactness of a modulation preserving operator implies compactness of the corresponding range operator. Furthermore, using equivalent definitions of trace and Hilbert schmidt norm, we obtain a necessary condition for a compact modulation preserving operator to be Hilbert Schmidt or of finite trace. Recall that an operator $T$ on a Hilbert space $\mathcal{H}$ is called compact, if $T(B)$ is relatively compact, where $B$ is the closed unit ball in $\mathcal{H}$. 
%Every compact operator $T$  has a canonical decomposition as follows
%\begin{equation} \label{co}
%T = \sum _{n} \lambda_{n} \langle . , e_{n} \rangle \sigma _{n} ,
%\end{equation}
%for some orthonormal sets $\{e_{n}\} $ and $\{\sigma_{n}\}$, where $\lambda_{n}$ is the sequence of singular values of $T$. In the case when $T$ is self adjoint and compact it can be shown that the decomposition of $T$ is as follows
%\begin{equation} \label{cop}
%T = \sum _{n} \lambda_{n} \langle . , e_{n} \rangle e _{n} ,
%\end{equation}
%for some orthonormal set $\{e_{n}\} $.
For more details on compact oparators we refer to usual text books related to operators, for example, \cite{M, Zhu}. The  following proposition gives a necessary condition for compactness of a modulation preserving operator. The proof is similar to \cite[Theorem 3.1]{KRsh}, so we  state the proposition without proof.
\begin{proposition}
Let $ W \subseteq L^{2}(G)$ be $\Lambda$-modulation invariant space with the range function $J$. Suppose that $U: W \longrightarrow L^2(G)$ is a $\Lambda$- modulation preserving operator with the range operator $R$. If $U$ is compct, then so is $R(\omega)$ for a.e. $\omega \in \Omega$.
\end{proposition}

Let $U$ be an operator on a Hilbert space $\mathcal{H}$ and $E$ be an orthonormal basis for $\mathcal{H}$. The Hilbert-Schmidt norm of $U$, denoted by $\Vert U \Vert _{HS}$, is defined as 
\begin{equation} \label{cop1}
\Vert U \Vert _{HS} = \left( \sum_{x\in E} \Vert Ux \Vert^2 \right) ^{\frac{1}{2}}.
\end{equation}
 The operator  $U$ is called of finite trace if $tr(U) < \infty$, in which
 \begin{equation} \label{cop2}
tr(U) = \sum_{x \in E} \langle Ux,x\rangle
\end{equation}
  is the trace of $U$. Note that the definitions are independent of the choice of orthonormal basis.
%
%
%
%
%
%We focus on the Schatten $p$-class of compact operators. For $p \in (0, \infty)$, a compact operator $T$ on a Hilbert space $\mathcal{H}$ is said to be in Schatten $p$-class $S_p$, if its singular value sequence $\{ \lambda _{n} \}$ belongs to $l^{p}$. One can see that $S_{p}$ is a Banach space with the norm
%\begin{equation*}
%\Vert T \Vert _{p} ^{p} := \sum_{n} \vert \lambda _{n} \vert ^{p} . 
%\end{equation*}
% In \cite[Chapter 2]{Zhu} it is shown that for a compact operator $T$ on a Hilbert space $\mathcal{H}$ we have 
%\begin{equation} \label{cop1}
% \Vert T \Vert _{2} ^2 =\sum _{k=1}^{\infty} \Vert Te_{k} \Vert ^{2} ,
%\end{equation}
%and if $T$ is positive 
%\begin{equation} \label{cop2}
% \Vert T \Vert _{1} =\sum _{k=1}^{\infty} \langle Te_{k} , e_{k} \rangle  ,
%\end{equation}
%where  $\{e_{k} \}$ is an orthonormal basis for $\mathcal{H}$.
The following lemma shows that we can benefit from Parseval frames instead of orthonormal bases in \eqref{cop1} and \eqref{cop2}. The proof is easy, so is omitted. For a complete proof see \cite[Lemma 3.2]{MRK}.
\begin{lemma} \label{l3.8}
 If $U$ is an operator on a  Hilbert space $\mathcal{H}$  and $F $ is a Parseval frame for $\mathcal{H}$, then $\Vert U \Vert _{HS} = \left( \sum_{y \in F} \Vert Uy \Vert^2 \right) ^{\frac{1}{2}}$. In particular if $U$ is positive,  $tr(U) = \sum_{y \in F} \langle Uy ,y\rangle$.
 \end{lemma}

We have the following proposition which states that a compact range operator is Hilbert Schmidt or of finite trace whenever the corresponding modulation preserving operator has the same properties.

\begin{proposition}
Suppose that $ W \subseteq L^{2}(G)$ is a $\Lambda$- modulation invariant space with the range function $J$. Let $U: W \longrightarrow W$ be a compact  modulation preserving operator with the range operator $R$.\\
(1) If $U$ is Hilbert Schmidt then so is $R(\omega)$ for a.e. $\omega \in \Omega$. \\
(2) If $U$ is positive and of finite trace then so is $R(\omega)$ for a.e. $\omega \in \Omega$.
% If  $ 1 \leq p < \infty$ and $U$ is in Schatten $p$- class $S_{p}$, then so is  $R(\omega)$ for a.e $\omega \in \Omega $. 
%In particular when $p=1$ and $U$ is positive,
%\begin{equation*}
%\Vert U \Vert _{1} = \int _{\Omega} \Vert R(\omega) \Vert _{1} \  d\omega .
%\end{equation*}

\end{proposition} 
\begin{proof}
First note that by Lemma \ref{t2.3} , $\{ M_{\lambda} \phi _{n}  :  \lambda \in \Lambda , \ n \in \mathbb{N} \}$ is a continuous Parseval frame for $W= \bigoplus _{n \in \mathbb{N}} M^{\Lambda}(\phi _{n})$ and hence by Proposition \ref{fr}, the set $\{ \tilde{Z}(\phi _{n}) (\omega) : \ n \in \mathbb{N} \} $ is a Parseval frame for $J(\omega)$, for a.e. $\omega \in \Omega$. Let $U$ be Hilbert Schmidt. Then
\begin{equation*}
\int_{\Lambda} \sum_{ n \in \mathbb{N}} \Vert M_{\lambda} U \phi _{n} \Vert ^2 dm_{\Lambda}(\lambda) = \sum_{ n \in \mathbb{N}} \int_{\Lambda} \Vert U(M_{\lambda} \phi _{n}) \Vert ^2 dm_{\Lambda}(\lambda)  = \Vert U \Vert _{HS}^2 < \infty .
\end{equation*}
Using the fact that $\tilde{Z}$ is isometry and Theorem \ref{t3.6}, we obtain 
\begin{align*}
\infty &>  \sum_{ n \in \mathbb{N}} \Vert U \phi _{n} \Vert ^2 \cr
&= \sum_{ n \in \mathbb{N}}\Vert \tilde{Z}  U \phi _{n}  \Vert ^2 \cr
&= \sum_{ n \in \mathbb{N}} \int _{\Omega} \Vert \tilde{Z}  U \phi _{n} (\omega) \Vert ^2 d \omega \cr
&= \int _{\Omega} \sum_{ n \in \mathbb{N}}\Vert \tilde{Z} U \phi _{n} (\omega) \Vert ^2 d \omega \cr
&= \int _{\Omega} \sum_{ n \in \mathbb{N}} \Vert R(\omega) (\tilde{Z}(\phi _{n})(\omega)) \Vert ^2 d\omega \cr
&= \int _{\Omega} \Vert R(\omega) \Vert _{HS} ^2 \  d\omega , 
\end{align*}
which shows that $R(\omega)$ is Hilbert Schmidt,  for a.e. $\omega \in \Omega$. 
If $U$ is positive and of finite trace, then by  the fact that $\tilde{Z}$ is isometry (in the second equality below), and Theorem \ref{t3.6} (in the third equality below), we have
\begin{align*}
\infty > tr(U)  &= \sum_{ n \in \mathbb{N}} \int_{\Gamma} \langle U M_{\lambda} \phi _{n} , M_{\lambda} \phi _{n} \rangle dm_{\Lambda}(\lambda) \cr
&= \sum_{ n \in \mathbb{N}} \int_{\Lambda} \langle \tilde{Z} U M_{\lambda} \phi _{n} , \tilde{Z} M_{\lambda} \phi _{n} \rangle dm_{\Lambda}(\lambda) \cr
&= \sum_{  n \in \mathbb{N}}  \int_{\Lambda} \int _{\Omega} \langle R(\omega )( \tilde{Z} M_{\lambda} \phi _{n} (\omega)) , \tilde{Z} M_{\lambda} \phi _{n} (\omega)  \rangle d \omega dm_{\Lambda}(\lambda)  \cr
&= \sum_{  n \in \mathbb{N}} \int_{\Lambda} \int _{\Omega} \langle R(\omega )( X _{\lambda} \tilde{Z}  \phi _{n} (\omega)) , X_{\lambda} \tilde{Z} \phi _{n} (\omega)  \rangle d \omega dm_{\Lambda}(\lambda) \cr
&= \int_{\Lambda} \sum_{  n \in \mathbb{N}}  \int _{\Omega} \langle R(\omega )( X _{\lambda} \tilde{Z}  \phi _{n} (\omega)) , X_{\lambda} \tilde{Z} \phi _{n} (\omega)  \rangle d \omega dm_{\Lambda}(\lambda) .
\end{align*}
So
\begin{align*}
\infty &> \sum_{ n \in \mathbb{N}} \int _{\Omega} \langle R(\omega )( \tilde{Z}  \phi _{n} (\omega)) , \tilde{Z} \phi _{n} (\omega)  \rangle d \omega  \cr
&=  \int _{\Omega} \sum_{ n \in \mathbb{N}} \langle R(\omega )(  \tilde{Z}  \phi _{n} (\omega)) , \tilde{Z} \phi _{n} (\omega)  \rangle d\omega, \ a.e. \ \omega \in \Omega.
\end{align*}
Thus  $R(\omega)$ is of finite trace, for a.e. $\omega \in \Omega$. 
%Let $U \in S_p$, for $p\geq 1$. Then by \cite[Lemma 1.25]{Zhu}, $ U^p \in S_1 $. By  the Stone-Weierstrass theorem, $U^p$ is a translation preserving operator. Hence the corresponding range operator $R(\omega) ^ p$, is in $S_1$. Again by \cite[Lemma 1.25]{Zhu}, we conclude that $ R(\omega) \in S_p$.
\end{proof} 
Our next proposition states that a necessary and sufficient condition for a modulation preserving operator to be isometric (self adjoint) is that its corresponding range operator is  isometric (self adjoint). The proof is similar to \cite[Propositions 3.4, 3.5]{KRsh} and  is omitted.
\begin{proposition}
Suppose that $ W \subseteq L^{2}(G)$ is a $\Lambda$- modulation invariant space with the range function $J$. Let $U: W \longrightarrow W$ be a compact  $\Lambda$- modulation preserving operator with the range operator $R$. Then\\
(1) $U$ is isometry if and only if so is $R(\omega)$ for a.e. $\omega \in \Omega$. \\
(2) $U$ is self adjoint if and only if so is $R(\omega)$ for a.e. $\omega \in \Omega$.
\end{proposition}
\begin{example}
Define $U  :   L^2(\Bbb{R}) \longrightarrow  L^2(\Bbb{R}) $ by $Uf(x) = f(x) + e^{2\pi i  x}f(x)$. Clearly $U$ is a modulation preserving operator. By Theorem \ref{t3.3} there exists a range operator $R$ such that for every $\phi \in L^2(\Bbb{R})$
\begin{equation*}
R(\omega) (\tilde{Z}\phi (\omega)) = (\tilde{Z} U) \phi (\omega) = \tilde{Z} \phi (\omega) + e^{2\pi i \omega} \tilde{Z}\phi (\omega) = (1+e^{2\pi i \omega})  \tilde{Z} \phi (\omega) .
\end{equation*}
Note that $R(\omega)$ is a multiplication operator (multiplication by $1+e^{2\pi i \omega}$) which is not compact. Notice that also $U$ is not compact. Moreover, $U^* f (x) = f(x) +e^{-2\pi i  x}f(x)$ and $R(\omega)^* (\tilde{Z} \phi (\omega)) =(1+e^{-2\pi i \omega})  \tilde{Z} \phi (\omega) $, where $f , \phi \in L^2(\Bbb{R})$.
\end{example}

\bibliographystyle{amsplain}

\end{document}